\documentclass[14pt,a4paper,oneside,english]{article}
\usepackage{graphicx,mathrsfs}
\usepackage[applemac]{inputenc}
\usepackage{amsmath,amsfonts,amssymb,amsthm,eufrak}
\usepackage{latexsym}
\usepackage{epstopdf,euscript}
\input xy 
\xyoption{all} 
\usepackage{babel}
\usepackage[nottoc]{tocbibind}
\usepackage{graphicx}
\graphicspath{{./figure/}}
\usepackage{latexsym}

\DeclareMathAlphabet{\mathpzc}{OT1}{pzc}{m}{it}

\numberwithin{equation}{section} 

\newcommand{\Z}{\mathbb{Z}} 
\newcommand{\C}{\mathbb{C}} 
\newcommand{\PP}{\mathbb{P}} 
\newcommand{\G}{\mathbb{G}} 

\DeclareMathOperator{\Spec}{Spec}

\newcommand{\aq}{\mathopen{[}}
\newcommand{\cq}{\mathclose{]}}

\newcommand{\beq}{\begin{equation*}}
\newcommand{\eq}{\end{equation*}}

\newcommand{\pma}{\begin{pmatrix}}
\newcommand{\pmat}{\end{pmatrix}}

\newcommand{\calC}{\mathcal{C}}
\newcommand{\tcalC}{\tilde{\calC}} 
\newcommand{\calD}{\mathcal{D}}
\newcommand{\tcalD}{\tilde{\calD}} 
\newcommand{\calO}{{\mathcal{O}}}
\newcommand{\calX}{\mathcal{X}}
\newcommand{\tcalX}{\tilde{\calX}}
\newcommand{\os}{\mathcal{O}_{S}} 
\newcommand{\osu}{\mathcal{O}_{S}^{*}} 
\newcommand{\ost}{\mathcal{O}_{S\cup T}} 
\newcommand{\calM}{{\mathcal M}}
\newcommand{\Pic}{{\operatorname{Pic}}}

\def\bbar#1{\setbox0=\hbox{$#1$}\dimen0=.2\ht0 \kern\dimen0 
\overline{\kern-\dimen0 #1}}

\newcommand{\uu}{\frac{u_{1}'}{u_{1}}}
\newcommand{\ud}{\frac{u_{2}'}{u_{2}}}
\newcommand{\la}{\frac{ \lambda '}{ \lambda }}

\theoremstyle{theorem}
\newtheorem{Th}{Theorem}[section]

\newtheorem{prop}[Th]{Proposition}
\newtheorem{lemma}[Th]{Lemma}
\newtheorem{conj}[Th]{Conjecture}

\theoremstyle{definition}

\newtheorem{remark}[Th]{Remark}

\newtheorem{example}[Th]{Example}

\newcommand{\Addresses}{{
  \bigskip
  \footnotesize

  \textsc{Department of Mathematical Sciences, Chalmers University of Technology and the University of Gothenburg,
SE-412 96 Gothenburg, Sweden}\par\nopagebreak
  \textit{E-mail address:} \texttt{tamos@chalmers.se}
}}

\title{Fibered Threefolds and Lang-Vojta's Conjecture over Function Fields}
\author{Amos Turchet}
\date{\today}
   
\begin{document}
\maketitle
\begin{abstract}
Using the techniques introduced in \cite{Corvaja2005} we solve the non-split case of the geometric Lang-Vojta Conjecture for affine surfaces isomorphic to the complement of a conic and two lines in the projective plane. In this situation we deal with sections of an affine threefold fibered over a curve, whose boundary, in the natural projective completion, is a quartic bundle over the base whose fibers have three irreducible components. We prove that the image of each section has bounded degree in terms of the Euler Characteristic of the base curve.
\end{abstract}

\section{Introduction}

Lang-Vojta\footnote{In this paper we quote the main conjecture with both names of Serge Lang and Paul Vojta, following the notation of \cite{HindrySilverman}, Conjecture F.5.3.6. The conjecture is nevertheless the same as the one of \cite{Corvaja2005} where it is denoted by \emph{Vojta Conjecture}. One can easily see that the statement is actually implied by the more general Vojta Conjecture on the height bound \cite{Vojta1987}, but this reformulation uses some ideas of Lang and, therefore, we decided to attribute it to both authors.} Conjecture (see \cite{Vojta1987}, \cite{CornellSilverman}, Chapter XV or \cite{HindrySilverman}, F.3.5 for a more basic introduction) is one of the most celebrated conjectures in Diophantine Geometry, generalizing to the logarithmic case the well known Bombieri-Lang's Conjecture (see \cite{Lang1983}). The conjecture predicts degeneracy of $S$-integral points in varieties of log-general type and, in the dimension 2 case, reads as follows:

\begin{conj}\label{LVnf}
  Let $\kappa$ be a number field, $S$ a finite set of places containing the archimedean ones. Let $X$ be a quasi-projective surface defined over $\kappa$ and let $\calX \to \Spec \calO_{\kappa,S}$ be a model of $X$ over the $S$-integers. Then, if $X$ is of logarithmic general type, $\calX(\calO_{\kappa,S})$ is not Zariski dense.
\end{conj}

Here a quasi projective variety is said to be of log-general type if there exists a desingularization $X_1 \to X$ and a compactification $\tilde{X}$ of $X_1$ such that the boundary divisor $D = \tilde{X} \setminus X_1$ has normal crossing singularities and $K_{\tilde{X}}+ D$ is big. If one starts with a presentation $X = Y \setminus D$ for a projective $Y$, this is equivalent to the existence of a log-resolution $Y',D'$ of the couple $Y,D$ such that $K_{Y'}+D'$ is big.
Since the spectrum of the ring of $S$-integers has dimension $1$, the model $\calX$ can be thought as an arithmetic threefold whose $S$-integral points correspond bijectively to sections of the structure map such that the image intersects the boundary over points of $S$. This geometric viewpoint can be carried over the so-called \emph{geometric} version of Lang-Vojta conjecture, i.e. where the number field is replaced by a function field of a curve. In these settings the conjecture becomes

\begin{conj}[Lang-Vojta]\label{Vconj}
Let $\tcalC$ be a smooth projective curve defined over an algebraically closed field $\kappa$ of characteristic 0, and let $S$ be a finite set of point of $\tcalC$. Let $X$ be a smooth affine surface defined over $\kappa(\tcalC)$. If $X$ is of log-general type, then there exists a bound for the degree of images of each section $\tcalC \setminus S \to X$ in terms of the Euler Characteristic of $\tcalC \setminus S$.
\end{conj}

The curve $\calC = \tcalC \setminus S$ plays the role of the spectrum of the $S$-integers in Conjecture \ref{LVnf} while morphisms $\tcalC \setminus S \to X$ corresponds to integral points over the function field of the curve $\tcalC$. Conjecture \ref{Vconj} asks for algebraic degeneracy of curves on surfaces of log-general type, i.e. for a bound of the degree depending on the genus and the number of points in the pullback of $D$. This property, after the work of Demailly \cite{Demailly1997}, is often called (weak) \emph{algebraic hyperbolicity}.

One of the most studied case is the one in which $\tilde X = \PP^2_\C$, when Conjecture \ref{Vconj} predicts weak algebraic hyperbolicity for the complement of a plane curve of degree at least 4. In this situation the split case, i.e. when the divisor $D$ and the surface $\tilde X$ are defined over the ground field, is known when $D$ has degree four and four irreducible components, and follows as an application of Stothers-Mason abc Theorem \cite{Brownawell1986} and \cite{Voloch1985}. In a different direction, and with different methods, the split case of the Conjecture has been proved in the case where $D$ has degree at least five by work of Chen \cite{Chen2004} and, independently, of Pacienza and Rousseau \cite{Pacienza2007}. In \cite{Corvaja2005}, Corvaja and Zannier proved, among other things, the three components and degree four split case and state a possible generalization of their main theorem when the divisor $D$ is not defined over the ground field. In this paper we deal with this situation that corresponds, geometrically, to the study of images of curves under sections of a fibered threefold.

We note that recently, in \cite{Corvaja2013}, Corvaja and Zannier vastly generalize the results of \cite{Corvaja2005} for affine surfaces admitting a finite dominant map to the two dimensional torus $\G_m^2$. We expect that an approach similar to the one presented here could lead to the extension of these results to the non-split case. It is also worth mentioning that the author in his Ph.D. Thesis \cite{TurThesis} completed the proof of Conjecture \ref{Vconj} in the split case for complements of very generic plane curves of degree at least four in $\PP^2$. However the tools used in the proofs do not seem to extend to the non-split case.

We will now describe precisely our result. Denote by $\kappa$ an algebraically closed field of characteristic 0 over which all algebraic varieties will be defined. Let $\calC$ be an affine curve (by which we always mean an integral separated scheme of finite type over Spec$(\kappa)$) with normalization $\tcalC \setminus S$, for a (unique) smooth complete curve $\tcalC$ and a finite subset $S$. The \textit{Euler characteristic} of the affine curve $\calC$ is defined as
\[
\chi (\calC) := \chi_{S} (\tcalC) = 2 g(\tcalC) - 2 + \sharp S,
\]
i.e. the Euler characteristic of the curve $\tcalC\setminus S$.

We recall briefly for the reader's convenience the ideas of \cite{Corvaja2005}. Given the affine surface $X = \PP^{2}\setminus D$, where $D$ is a divisor consisting of a conic and two lines in general position, one can assume that one of the lines is the line at infinity. In such a contest one is led to the study of affine curves in the complement of a line and a conic in the affine plane. Morphisms from an affine curve $\tcalC \setminus S$ to such a surface can be expressed as $f: P \mapsto (u_1(P),y(P))$, for a couple of rational functions $u_1,y$ of $\tcalC$, with the property that the zeros and the poles of $u_1$, and the poles of $y$ are contained in $S$. Using explicit equations for $D$, one can write down an equation satisfied by these functions and another unit $u_2$ that looks as follows  
\[
y^{2} = u_{1}^{2}+ \lambda u_{1} + u_{2} + 1.
\]
Bounding the degree of the morphism $f$ is then equivalent to bound the height of solutions to the previous equation. Hence the problem relies on solving the equation in so-called $S$-units $u_1,u_2$ and $S$-integer $y$. For this the authors consider a specific differential form on the curve with respect to which the previous equation can be differentiated. Then one can prove that this new equation should have many zeros in common with the previous one and hence, using a gcd argument for $S$-units, its solutions have either bounded height or fulfill a dependence relation. In both cases one can conclude the proof of Conjecture \ref{Vconj}.

We note, passim, that the differential equation obtained by the use of the differential form is of particular interest in the context of relating these results to hyperbolicity problems: notably this equation is strictly related with the problem of finding sections of logarithmic jet differentials, a feature which is essential in results obtained in the complex analytic case (see \cite{Mourougane2012} and \cite{Rousseau2009} for more details).

Our goal in this paper is to generalize this situation to the so called non-split case, i.e. the case of Lang-Vojta Conjecture for the complement of a conic and two lines in $\PP^2$, where now the divisor has field of definition strictly bigger than the ground field. As in the constant case, we are going to reduce the problem to solve a Diophantine equation and bound the height of its solutions using Corvaja and Zannier method. In our case, the equation that describes this setting reads as follows:
\begin{equation}\label{*}
y^{2} = u_{1}^{2} + \lambda(P) u_{1} + u_{2} + 1.
\end{equation}
Here again $y$ is a $S$-integer, $u_1,u_2$ are $S$-units, and $\lambda$ will be a rational function on the curve $\tcalC$. We note that this equation is precisely the same considered in \cite{Corvaja2005}, where now the polynomial in the right hand side has non-constant coefficients. Geometrically, this corresponds to the data of an (affine) threefold $X$, fibered over the curve $\calC$, where each fiber is isomorphic to $\PP^{2} \setminus D$ and $D$ is a divisor consisting of a conic and two lines. Each solution of the equation (\ref{*}) gives a section of the fibration $X \to \calC$. This could be formalized using the theory of integral models, where $S$-integral points of a variety defined over a function field of a curve corresponds to sections of the structure map given by the model, such that the inverse image of the divisor $D$ is supported on $S$. However we are not going to use this approach in this article.

The situation we will consider is made explicit in the following diagram:
\begin{equation} \label{diagram}
\xymatrix{ & &X \ar[d]^{\pi} & \\ &  &\calC \ar[r]^{ \lambda} \ar@/^1pc/[u]^{ \sigma}& \PP^{1}}
\end{equation}
The parameter $\lambda(P)$ will be a function depending on the divisor on the fiber over $P$ and $ \sigma$ will be a section of the projection $\pi$. We observe that in \cite{Corvaja2005}, the morphisms considered by Corvaja and Zannier from the affine curve $\calC = \tcalC \setminus S$ to $\PP^{2}\setminus D$ can bee seen as sections of the trivial ($\PP^{2} \setminus D$)-bundle over the curve $\calC$. Here the trivial bundle is replaced by a fibration, in which the divisor at infinity in the fibers is moving. Moreover, generalizing the situation of the constant case, the three irreducible components of the divisor $D = D_{P}$ are not supposed to be in general position for every $P \in \tcalC$ (although we need some restriction on the ``degeneracy'' of the divisor).

Our main result is the following

\begin{Th}\label{th:deg}
Let $\tcalC, S, X$ as above. Let $ \sigma : \calC \to X$ be a non constant section for the fibration $\pi : X \to \calC$, where each fiber is isomorphic to $\PP^{2} \setminus D$. Then there exist effectively computable constants $C_1,C_2$ such that, in a suitable projective embedding of $X$, the curves $\sigma(\tcalC)$ verify
\[
\deg \sigma(\tcalC) \leq C_1 \cdot \chi_S (\tcalC) + C_2
\]

\end{Th}

In section \ref{sec:threefolds} we are going to explicitly determined the constant $C_1$, $C_2$ using one natural compactification of the affine variety we are considering. This will result in having the constants depending explicitly on the height of the rational function $\lambda$.

\begin{remark}
Readers familiar with \cite{Corvaja2005} will easily recognize a lot of similarity in the proof: one could also say that, a part from some geometric description in the first two section, we basically cover the same steps carrying out the computations in this different framework. We nevertheless decided to include each intermediate step, even the one more close to Corvaja and Zannier ones, in order to make this article accessible to all readers.
\end{remark}

\paragraph*{Acknowledgments.}
The results presented in this article are part of my Ph.D. thesis supervised by Pietro Corvaja: I take the opportunity to thank him for having introduced me to this fascinating problems, for his mathematical guidance and for several suggestion on this paper	. I also thank Dan Abramovich, Pietro De Poi, Barbara Fantechi and Francesco Zucconi for stimulating discussions which helped to improve the presentation.

\paragraph*{Notations.}
We now set the notation for the proofs and we will follow as much as possible the one introduced in \cite{Corvaja2005}. From now on $\tcalC$ will be a smooth complete algebraic curve defined over $\kappa$ of genus $g = g(\tcalC)$ and $S \subset \tcalC$ will be a finite set of points of $\tcalC$. We shall denote by $\os$ the ring $\kappa\aq \tcalC \setminus S \cq$ of regular functions on the affine curve $\calC = \tcalC \setminus S$ and we will call it the ring of $S$-integers (elements of $\os$ are just rational functions on $\tcalC$ with poles contained in the set $S$). The group of units $\osu$ is called the group of $S$-units, i.e. rational functions on $\tcalC$ with poles and zeros contained in $S$. As before we define the Euler characteristic of the affine curve $\calC = \tcalC \setminus S$ as $\chi_{S}(\tcalC) = \chi( \tcalC \setminus S) = 2 g(\tcalC) - 2 + \sharp  S$.

For any rational function $a \in \kappa(\tcalC)$ we define its height to be its degree as a morphism to $\PP^{1}$. More explicitly, if we associate to each point $v \in \tcalC$ a discrete valuation of the function field $\kappa(\tcalC)$ trivial on $\kappa$ and normalized, so that its value group is $\Z$, then the height of $a$ can be expressed as 
\[
H_{\tcalC} (a) := \sum_{v \in \tcalC}\, \max \{ 0, v(a) \},
\]
where we denoted the valuation with the same letter $v$. 
Given a dominant morphism of smooth irreducible projective curves $\tcalD \to \tcalC$, an element $a \in \kappa(\tcalC)$ can be viewed as a rational function on $\tcalD$ via the inclusion $\kappa(\tcalC) \subset \kappa(\tcalD)$ given by the morphism. The height of $a$ with respect to $\tcalD$ verifies
\[
H_{\tcalD} (a) = \aq \kappa(\tcalD) : \kappa(\tcalC) \cq \cdot H_{\tcalC} (a).
\]
%

\section{Configuration of a conic and two lines}\label{sec:config}
In this section we will analyze configurations of a conic and two lines in $\PP^{2}$. Our aim is to prove that a moduli space for equivalence classes of these divisors is of dimension one.
\begin{figure}[ht]
\begin{center}
\includegraphics[width =.38\textwidth]{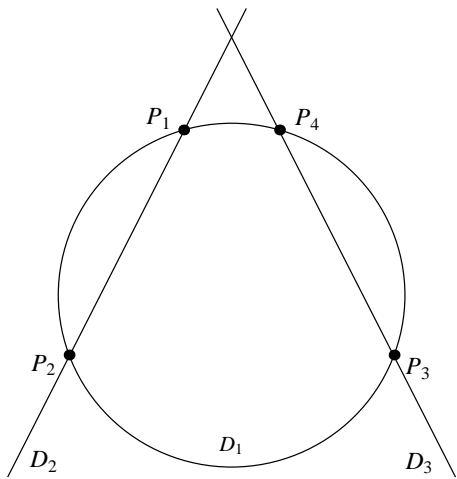}
\caption{Configuration of a conic and two lines in general position}
\label{conicandlinefig}
\end{center}
\end{figure}
Let $D$ be the sum of a smooth conic $D_{1}$ and two distinct lines $D_{2},D_{3}$ in $\PP^2$ defined over $\kappa$. This divisor has at most five singular points, four of which lie on the conic, i.e. the four points of intersection between $D_{1}$ and $D_{2}+D_{3}$; these points are distinct in the case when $D_{1},D_{2},D_{3}$ are in general position, i.e. $D$ has normal crossing singularities. We want to characterize completely isomorphism classes of such divisors.

First we observe that each class possesses a representative with a fixed conic $\bbar{D_1}$ as component of degree two. Hence the problem can be reduced to study isomorphism classes of unordered couples of lines not tangent to $\bbar{D_1}$ whose intersection is not on the conic. One of such divisor is visible in figure \ref{conicandlinefig}. 
Secondly one can notice that the problem is equivalent to the study of classes of fourples of points on $\PP^1$, via the isomorphism between the conic and $\PP^1$ (here we fix the ordering of the lines such that the line $D_2$ is the one passing through the first two points, and the line $D_3$ passes through the last two points): we say that two fourples are in the same class if the corresponding unordered couple of unordered couples of points coincide. This means precisely that the divisors consisting of the fixed conic $\bbar{D_1}$ and the lines $D_2$ and $D_3$ passing through the ordered points are isomorphic.

A moduli space for our problem will therefore be represented by a scheme with a map from $\calM_{0,4} \cong \G_m \setminus \{1\}$, where the last isomorphism is given by the cross-ratio. This follows form the easy observation that fourples with the same cross-ratio are in the same class. However, the map from $\calM_{0,4}$ will not be injective: as an example consider the following two fourples $( P_{1},P_{2},P_{3},P_{4})$ and $( P_{2}, P_{1},P_{3},P_{4} )$, which clearly are in the same class, since
\[
\{\{ P_{1}, P_{2} \}, \{P_{3}, P_{4}\}\} = \{\{ P_{2}, P_{1} \}, \{P_{3}, P_{4}\}\},
\]
but their cross-ratios are inverse of each other. Hence we expect the map from the moduli of fourples of points in $\PP^1$ to our moduli space to be many-to-one. To describe it more precisely we use the following basic lemma:

\begin{lemma}\label{cross-ratios}
 Let $P=(P_1,\dots,P_4)$ and $Q=(Q_1,\dots,Q_4)$ be two fourples such that $\underline{P} \neq \underline{Q}$ in $\calM_{0,4}$, i.e. they don't have the same cross-ratio. Then, if they are in the same class, i.e. they represent the same divisor, then there exists a permutation $\sigma$ such that, $\underline{\sigma(P)}=(P_{\sigma(1)},\dots,P_{\sigma(4)})=\underline{Q}$, as elements of $\calM_{0,4}$.
\end{lemma}

Hence we are reduced to calculate which permutations of four points give rise to isomorphic configurations of divisors. We can then consider the action of the permutation group $S^{4}$ on an ordered set of four points in the projective line, i.e. an element of $(\PP^{1})^{4}$; an easy case by case analysis shows that the subgroup of $S^4$ under which a class of fourples is invariant is $G = \langle (12),(13)(24),(14)(23)\rangle$. Hence, by classic properties of the cross-ratio, the only non trivial action of $G$ in $\calM_{0,4}$ comes from elements of the form $(12)\cdot G$. This corresponds precisely to the example made before, in which the first two points are flipped.

In order to completely describe isomorphism classes of degree four and three components divisors in $\PP^2$ it is sufficient to define a map
\[
\lambda' : \left\{ \begin{matrix} \text{degree four and three} \\ \text{components divisor in } \PP^2 \end{matrix} \right\}   \longrightarrow \PP^1,
\]
constant on isomorphic divisors. By the description given above we obtain a natural 2:1 map, from the moduli space $\calM_{0,4}$ to the moduli space of degree four and three components divisors: this map simply associates to each fourple the two lines passing to the four points. With abuse of notation we indicate as $\lambda'$ the composition
\[
\lambda' : \calM_{0,4} \to \left\{ \begin{matrix} \text{degree four and three} \\ \text{components divisor in } \PP^2 \end{matrix} \right\}  \to \PP^1,
\]
which is defined as
\begin{equation}\label{lambda}
  \lambda' (P_{1},P_{2},P_{3},P_{4}) = \frac{\beta\left(P_{1},P_{2},P_{3},P_{4}\right)^{2} + 1}{\beta\left( P_{1},P_{2},P_{3},P_{4}\right)} -2\notag = \beta(P_{1},P_{2},P_{3},P_{4}) + \dfrac{1}{\beta(P_{1},P_{2},P_{3},P_{4})} - 2
\end{equation}
where $\beta$ is the cross-ratio. We observe that the involution $(12)$ inverts the cross-ratio hence the function $\lambda'$ is constant on fourples in the same class. The presence of the $-2$ in the formula comes from the fact that we want the zeros and the poles of $\lambda'$ to contain all the fourples corresponding to divisors which fail to be normal-crossing; these are precisely the ones in which the cross-ratio is either not defined, or takes values $0,\infty$ and $1$. 

In our computation we will assume that the set $S$ contains all the points where the fiber has boundary divisor defined by a fourple that is either a pole or a zero of $\lambda'$.
This will, in particular, allow us to apply Theorem \ref{Zannier} and to obtain a stronger result, given the fact that our final result will not depend on the negativeness or positiveness of the Euler characteristic of the base curve. Namely our bound on the degree of the curves will have a constant term depending on the height of $\lambda$ but all the constants will not depend on the sign of the characteristic. At the same time, this is not a strong restriction because $\lambda'$ will be a datum of the variety we want to deal with and hence it does not depend on the method used for the proof (see next section for precise definition of $\lambda$ and its role in the definition of the threefold $X$).

With abuse of notation we will sometimes indicate the value $\lambda'(P_{1},\dots,P_{4})$ as $\lambda'(D)$ where the configuration of $D$ is defined by the points $P_{1},\dots,P_{4}$ on the conic $D_{1}$.

\section{Affine threefolds}\label{sec:threefolds}

We are interested in a specific class of affine threefolds fibered over affine curves which generalizes the trivial $\PP^2 \setminus D$-bundle considered in the split case. More in detail we will consider the following class of affine threefolds:

\begin{quote}\label{threef}
$(\star)$ $X$ is an affine threefold fibered over the affine curve $\calC$ such that the completion of the fibration is the trivial $\PP^{2}$-bundle over $\tcalC$. Every fiber $\pi^{-1}(P)$ for a point $P \in \tcalC$ is of the form $\PP^{2} \setminus D_{P}$ where $D_{P}$ is a divisor of $\PP^{2}$ of degree four consisting of an irreducible conic and two lines such that there are at least four distinct singular points. If the point $P$ is in $\calC$ then the function $\lambda'$ is regular on $D_P$, i.e. $D_P$ has normal crossing singularities.
\end{quote}

\noindent A picture of this situation can be seen in the diagram \ref{diagram}. 

\begin{example}
Consider a plane smooth affine curve $\calC$ in $\PP^2_\C$. For each point $P \in \calC$ let $t_P$ denote the tangent line to $\calC$ at $P$. This defines a fibration over $\calC$ in the following way: over a point $P \in \calC$ let $\calX_P$ be the complement in $\PP^2$ of the divisor formed by a fixed quadric $D_1$, the line at infinity $D_\infty$ (assuming a choice of coordinate has been made) and the line $D_{3,P} = t_P$. A picture of this situation can be seen in Figure \ref{fig:threefold}.
\begin{figure}[h]
\caption{Fibered threefold}\label{fig:threefold}
\centering
\includegraphics[width=.8\textwidth]{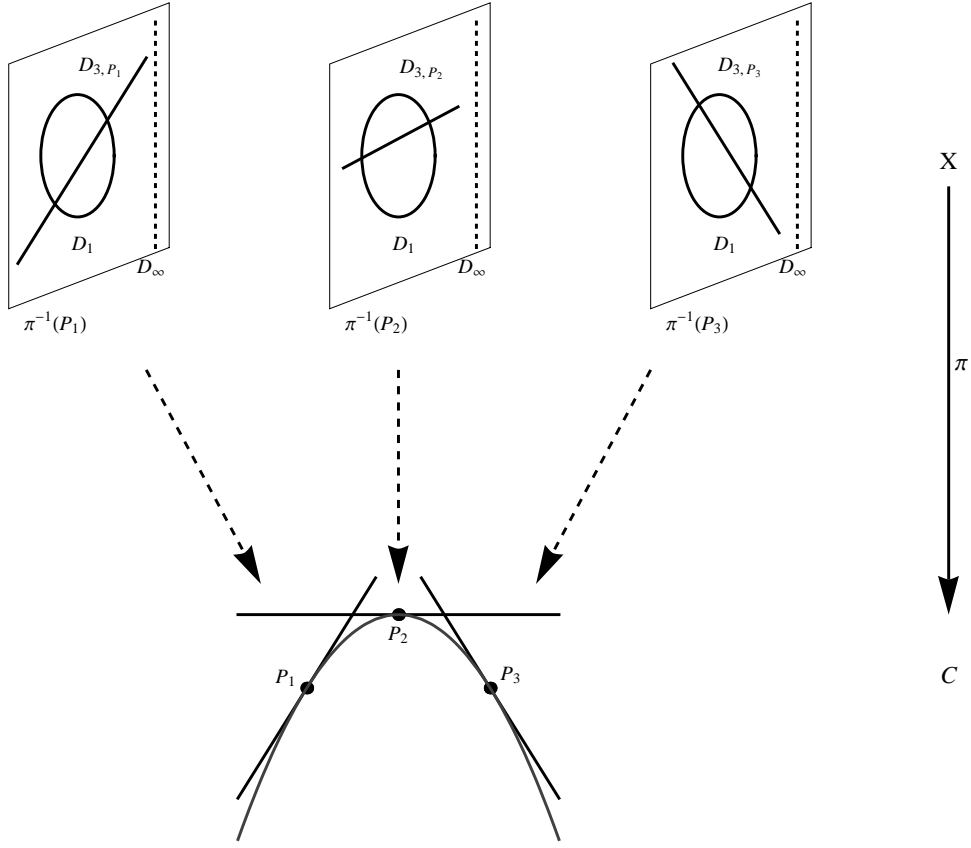}
\end{figure}
The threefold
\[
\calX = \bigcup_{P \in \calC} \calX_P \to \calC,
\]
can be seen as a surface defined over the function field of the completion of the normalization $\tcalC$ of $\calC$ where a point of the surface $P \in \calX(k(\tcalC))$ corresponds to a section $\sigma_P: \tcalC \to \tcalX$ such that $\sigma_P^{-1}(D) \subset \tcalC \setminus \calC$. In particular in the case in which the divisor
\[
D_P = D_1 + D_\infty + D_{3,P}
\]
has normal crossing for all $P \in \tcalC \setminus S$, being $\deg D = 4$, each fiber is of log-general type. Therefore, Theorem \ref{th:deg} can be applied giving a bound for the degree of images $\sigma_P(\tcalC)$ as expected by Conjecture \ref{Vconj}.
\end{example}

It follows from the definition of the class $(\star)$ that giving such a threefold is equivalent to giving a rational function
\[
\lambda : \tcalC \dashrightarrow \PP^1,
\]
which associates to a point $P \in \tcalC$ a point of $\PP^1$ viewed as the value of the function $\lambda'(D_P)$, i.e. $\lambda(P)$ specifies the isomorphism class of the divisor $D_P$ in the fiber over $P$.
More in detail each of the affine threefolds we are going to study can be described as follows: we can naturally embed $X$ inside $\tilde{X} := \tcalC \times \PP^2$ and denote by $p_1: \tilde{X} \to \tcalC$ and $p_2 : \tilde{X} \to \PP^2$ the two projections. Then the fibration $X \to \calC$ is uniquely determined by a line bundle $\mu \in \Pic(\tcalC)$ and the choice of a divisor $D \in \lvert p_1^*(\mu) \otimes p_2^*(\calO_{\PP^2}(4)) \rvert$, see diagram below. 
\begin{equation}\label{bundles}
\xymatrix@C=5pt{									&p_1^*(\mu) \otimes p_2^*(\calO_{\PP^2}(4)) \ar[d] & \\
			X\ \ \ar[d]^{\pi} \ar@{^{(}->}[r]	&\tcalC \times \PP^2 \ar[d]^{p_1} \ar[r]^{p_2}	& \PP^2 \\
			\calC\ \  \ar@/^1pc/[u]^{ \sigma} \ar@{^{(}->}[r] &\tcalC }
\end{equation}
We stress that this general construction does not guarantee that the threefold $X$ belongs to the class $(\star)$ since not every divisor in the linear system gives rise to a fibration where all the fibers $D\vert_{\pi^{-1}(P)}$ have three components. Conversely, for each such a threefold, one can obtain a description as the one given above for a suitable section of $\calO_{\PP^2}(4)$. This in particular implies that, if $X$ belongs to $(\star)$, on every fiber the divisor is determined, up to isomorphism, by the value of the function $\lambda$ defined above, i.e. over every point $P \in \calC$ the fiber is uniquely specified by the value of $\lambda'$ on the singular points of $D_{P}$. In other words if $X \to \calC$ is one of the threefolds we are interested in, the datum of the map $\lambda$ completely describes the geometry of $X$. 

In the following section we will prove that every threefold in the class $(\star)$, characterized by a non constant rational map $\lambda: \tcalC \dashrightarrow \PP^1$, has images of sections with bounded degree in terms of the Euler Characteristic of the base curve.


\section{Sections of the fibered threefold}
From now on we will work on an affine algebraic variety of dimension three verifying condition ($\star$). We will denote by $D_{P}$ the divisor defined on the fiber over the point $P$ (or simply $D$ where the point will be clear) and its three irreducible components will be indicated by $D_{1}$ (the conic) and $D_{2},D_{3}$ (the two lines). $ \lambda : \tcalC \to \PP^{1}$ will denote the morphism $\lambda(P) := \lambda'(D_{P})$ and we will suppose that $S$ contains all its zeros and poles. We begin by proving the following:

\begin{lemma}\label{eql}
Let $\tcalC,S$ as before and let $\pi:X \to \calC$ an affine threefold verifying the condition $(\star)$. Let $ \sigma : \calC \to X$ be a section of $\pi$. Then, possibly after passing to a cover of $\calC$ of degree 2, there exist $S$-units $u_{1},u_{2} \in \osu$ and an $S$-integer $y \in \os$ satisfying
\begin{equation}\label{eq}
y^{2} = u_{1}^{2} + \lambda u_{1} + u_{2} + 1,
\end{equation}
and such that $\deg \sigma(\calC) \leq H_{\tcalC}(u_1) + H_{\tcalC}(y)$.
\begin{proof}
From condition $(\star)$ it follows that, at most after choosing homogeneous coordinates, we can considered affine coordinates $(x,y)$ in every fiber with respect to the line $D_{2}$ that will be the line at infinity $x_{0}=0$. In this coordinate system the line $D_{3}$ has equation $x=0$ and the conic $D_{1}$ has equation $y^{2}=x^{2}+ \lambda x +1$(by Tsen Theorem the equations of the two lines can always be chosen in the desired way. However it could happen that the conic cannot simultaneously be taken to have the former form. In this case we consider a double cover of $\tcalC$ where this holds and perform all the computation in the cover. This will affect only the numerical constants involved in the computation by, at most, a factor of 2). Now we turn our attention to the section $\sigma : \tcalC \setminus S \to X$. In our setting $ \sigma$ can be written as
\[
\sigma (P) = (x(P),y(P),P) \in \pi^{-1}(P) \cong \PP^{2} \setminus D_{P}.
\]
Now it is a general fact that such a morphisms has degree bounded by the height of its components: indeed the degree of $\sigma$ is the number of intersection points with a generic hyperplane in a projective space where $\sigma(\tcalC)$ is embedded and this number is bounded by the sum of the heights of the components $x$ and $y$. This proves that $\deg \sigma(\calC) \leq H_{\tcalC}(u_1) + H_{\tcalC}(y)$ where $u_1 := x$.
The fact that the image $ \sigma(P)$ avoids the line $D_{2}$ means that the function $u_{1} := x \in \osu$ is a unit and $y \in \os$ is a regular function on the affine curve $\calC$. Moreover, the condition that the image of $ \sigma$ avoids also the conic $D_{1}$ in every fiber means that we can define another $S$-unit $u_{2}$ where
\[
u_{2} = y^{2} - u_{1}^{2} - \lambda u_{1} - 1.
\]
Hence the units $u_{1},u_{2}$ and the $S$-integer $y$ verify equation (\ref{eq}) concluding the proof.
\end{proof}
\end{lemma}

We will now work with equation (\ref{eq}) in order to describe its solutions. Our goal is to prove the following

\begin{Th}\label{mainth}
With the notation above, every solution $(y,u_{1},u_{2}) \in \os \times (\osu)^{2}$ of equation (\ref{eq}) satisfies one of the following conditions:
\begin{description}
\item[(i)] a sub-sum on the right term of (\ref{eq}) vanishes;
\item[(ii)] $u_{1},u_{2}$ verify a multiplicative dependence relation of the form $u_{1}^{r} \cdot u_{2}^{s} = \mu$, where $ \mu \in \kappa^*$ is a scalar and $r,s,$ are integers, non both zeros such that $\max\{r,s\} \leq 5$;
\item[(iii)] the following bound holds:
\[
\max\{H_{\tcalC}(u_{1}),H_{\tcalC}(u_{2})\} \leq  2^{12} \big( 58 \chi_{S}(\tcalC) + 28 H_{\tcalC}(\lambda) \big) + 1 H_{\tcalC}(\lambda).
\]
\end{description}
\end{Th}

We will now follow the proof, given by Corvaja and Zannier, of the constant case deepening the differences between our situation and the ideas of \cite{Corvaja2005}.
Firstly we observe that we can define a differential form on $\tcalC$ such that we can speak of derivatives of rational functions. In particular we have the following (this is Lemma 3.5 and Lemma 3.6 of \cite{Corvaja2005})

\begin{lemma}\label{derivative}
There exists a differential form $ \omega \in \tcalC$ and a finite set $T \subset \tcalC$ of cardinality $\sharp  T = \max\{0, 2g(\tcalC) -2 \}$ such that for every $u \in \osu$ there exists an $(S \cup T)$-integer $ \theta_{u} \in \ost$ having only simple poles such that
\begin{equation}
\frac{d(u)}{u} = \theta_{u} \cdot \omega \qquad \qquad H_{\tcalC}( \theta_{u}) \leq \chi_{S}(\tcalC).
\end{equation}
Moreover if $a \in \os$ then there exists an $a' \in \ost$ such that
\[
d(a) = a' \cdot \omega \qquad \qquad H_{\tcalC}(a') \leq H_{\tcalC}(a) + \chi_{S}(\tcalC).
\]
\end{lemma}

For the proof of this lemma we refer again to \cite{Corvaja2005}; we just notice that Lemma \ref{derivative} refers to the curve only, without any reference to the bundle, and hence can be applied in all the cases under consideration. From now on the differential form $ \omega$ and the finite set $T$ will be fixed and, for a rational function $a \in \kappa(\tcalC)$ we will denote by $a'$ the only rational function such that $d(a) = a' \cdot \omega$.

We consider now the derivative of a polynomial $A \in \kappa\aq X,Y\cq$ calculated in a point $u_{1},u_{2}$ for some $S$-units $u_{1},u_{2}$. One can prove that (see \cite{Corvaja2005} Lemma 3.7)
\[
(A(u_{1},u_{2}))' = B(u_{1},u_{2}),
\]
where
\[
B(X,Y) = \uu \cdot X \, \frac{\partial A}{ \partial X}(X,Y) + \ud \cdot Y\, \frac{\partial A}{\partial  Y} (X,Y).
\]
We will use this identity in order to deal with equation (\ref{eq}).

\begin{lemma}\label{AB}
Let 
\begin{align}
A(X,Y) &= X^{2} + \lambda X + Y + 1 \notag \\
B(X,Y) & =  2\ \uu\ X^{2} + \lambda\, \bigg( \uu \, +\, \la\bigg)\,X\, + \,\ud\ Y\, 
\end{align}
polynomials in $\ost(\tcalC)\aq X,Y\cq$, and let $F(X) \in \ost\aq X\cq$, $G(Y) \in \ost \aq Y \cq$ be the resultants of $A(X,Y),B(X,Y)$ with respect to $Y$ and $X$, i.e. the polynomials
\begin{align}
F(\mathbf{X}) = Res_{Y}(A,B) &= \mathbf{X}^{2} \bigg( 2\uu - \ud \bigg) + \mathbf{X} \bigg(\uu -\ud + \la \bigg) \lambda - \ud
\\
\ & \  \notag
\\
G(\mathbf{Y}) = Res_{X}(A,B) & = \mathbf{Y^{2}} \bigg( 2 \uu - \ud \bigg)^{2} + \mathbf{Y} \bigg[ \bigg( \uu \bigg)^{2} ( 8 - \lambda^{2})  + \uu \ud ( \lambda^{2} - 4) + \lambda^{2} \la \bigg( \la - \ud \bigg) \bigg] +  \notag
\\ & \qquad \qquad + \bigg( \uu \bigg)^{2} (4 - \lambda^{2}) + \lambda^{2} \bigg( \la \bigg)^{2}.
\end{align}
Then for every solution $(y,u_{1},u_{2}) \in \os \times (\osu)^{2}$ of (\ref{eq}) we have
\begin{align*}
y^{2} &= A(u_{1},u_{2}), \\
2 y y' &= B(u_{1},u_{2}).
\end{align*}
Moreover the $S$-integer $y$ divides both $F(u_{1})$ and $G(u_{2})$ in the ring $\ost$.
\begin{proof}
Obviously equation (\ref{eq}) is exactly $y^{2} = A(u_{1},u_{2})$. Moreover $A(u_{1},u_{2})' = B(u_{1},u_{2})$ so we have $2yy' = B(u_{1},u_{2})$ as wanted.

For the second fact we observe that, for the general theory of resultants, $F$ and $G$ are linear combinations of $A$ and $B$ with coefficients that are polynomials in $\ost$, concluding the proof.
\end{proof}
\end{lemma}

Our next step is to factor the polynomials $F(X),G(Y)$ in a suitable finite field extension of $\kappa(\tcalC)$; this extension will be a function field $\kappa(\tcalD)$ for a cover $\tcalD \to \tcalC$. Besides, we will estimate the Euler characteristic of the curve $\tcalD$. From now on we will suppose that the leading and the constant term of the polynomial $F(X),G(Y)$ are both non zero.

\begin{lemma}\label{cover}
Given $F,G,\tcalC,S,T$ as before, there exists a cover $\tilde{\calD} \to \tcalC$, of degree less or equal to four, such that the Euler characteristic of $\tilde{\calD} \setminus U$ verifies
\begin{equation}\label{chiD}
\chi_{U}(\tilde{\calD}) \leq 53 \chi_{S}(\tcalC) + 28 H_{\tcalC}(\lambda) + 5 \cdot max \{ 0, 2g(\tcalC) - 2 \},
\end{equation}
where $U$ is the set formed by the pre-images of the zeros of the leading and constant coefficients of $F$ and $G$ and the pre-images of $S$ and $T$.

\begin{proof}
Our goal was to factor $F(X)$ and $G(X)$, so we define the cover $p : \tilde{\calD} \to \tcalC$ by the property that $\kappa(\tilde{\calD})$ is the splitting field of $F(X)\cdot G(X)$ over $\kappa(\tcalC)$. From this definition it is straightforward that $\deg p$ is at most four, because $\kappa(\tilde{\calD})$ is generated over $p^{*}(\kappa(\tcalC))$ by the square roots of the discriminants of the two polynomials (recall that $F(X)$ and $G(X)$ both have degree 2).

We will now bound the Euler characteristic of $\tilde{\calD} \setminus U$ via the Riemann-Hurwitz formula; for this goal we need an estimate of the ramification points of the cover $p$. First of all we notice that the ramification points are all contained in the zeros and poles of the discriminants; moreover at any point the ramification index is at most two. The poles are contained in $S \cup T$ and the number of zeros of the discriminants is bounded by their heights. The discriminant of $F(X)$ is

\begin{equation}\label{discrF}
Discr (F(X))   =  \bigg( \ud \bigg)^{2} ( \lambda^{2} -4) + \bigg( \ud \bigg) \bigg( 8 \uu - 2 \uu \lambda^{2}  - 2 \lambda \lambda' \bigg) + \bigg( \lambda \uu + \lambda'^2 \bigg)^{2},
\end{equation}
so its height (which can be estimated counting its possible poles) is bounded by
\begin{align*}
H_{\tcalC} ( Discr(F(X))) \leq 2 H_{\tcalC} \bigg( \ud \bigg) +  2 H_{\tcalC} \bigg( \uu\bigg) +  2 H_{\tcalC} ( \lambda' ) + 2 H_{\tcalC} ( \lambda) \leq 6 \chi_{S} (\tcalC ) + 4 H_{\tcalC} ( \lambda ).
\end{align*}
Analogously, we can look at the discriminant of $G(X)$

\begin{align}\label{discrG}
Discr(G(X)) &= \bigg( \ud \bigg)^{2} \bigg[ \bigg(\uu \bigg)^{2} \lambda^{2} (4 - \lambda^{2}) + \uu \lambda \lambda' (8 - 2 \lambda^{2}) +  \lambda'^{2} ( \lambda^{2} - 4) \bigg]+ \notag \\ &
+ 2\ \ud   \bigg[ \bigg(\uu \bigg)^{3} \lambda^{2} (4 - \lambda^{2}) \bigg(\uu \bigg)^{2} \lambda \lambda'( \lambda^{2} - 8) + \uu \lambda'^{2} ( 4 + \lambda^{2}) - \lambda \lambda' \bigg] + \notag \\ &
+ \bigg(\uu \bigg)^{4} \lambda^{2} - 2 \bigg(\uu \bigg)^{2} \lambda^{2} \lambda'^{2} + \lambda'^{4}
\end{align}

and bound its height in the same way, obtaining that $H_{\tcalC} ( Discr(G(X)))$ is bounded above by

\begin{align*}
2 H_{\calC} \bigg( \ud \bigg) + 4 H_{\tcalC} \bigg( \uu\bigg) + 4 H_{\tcalC} ( \lambda' ) + 4 H_{\tcalC} ( \lambda) \leq 10 \chi_{S} (\tcalC ) + 8 H_{\tcalC} ( \lambda ).
\end{align*}

Therefore the number of ramification points is at most
\[
\sharp (S\cup T) + 16 \chi_{S}(\tcalC) + 12 H_{\tcalC} (\lambda).
\]
We can now apply the Riemann-Hurwitz formula

\begin{equation}\label{RU}
2g(\tilde{\calD}) - 2 = (\deg p)(2g(\tcalC) - 2) + \sum_{P \in \tilde{\calD}} (e_{P} - 1).
\end{equation}

Here $e(P)$ is the ramification index of $p$ at $P$ and thus $(e_{P} - 1)$ is either zero or one. By  above estimate of the ramification points of $p$ we obtain that

\begin{equation}\label{ram}
\sum_{P \in \tilde{\calD}} (e_{P} - 1) \leq \sharp (S\cup T) + 16 \chi_{S}(\tcalC) + 12 H_{\tcalC} (\lambda).
\end{equation}

Consider now the set $U \subset \tilde{\calD}$ introduced in the statement of the Lemma. We have that
\[
\sharp  U \leq \aq \kappa(\tilde{\calD}) : p^{\star}(\kappa(\tcalC)) \cq \cdot \sharp p(U).
\]
From this inequality and from (\ref{RU}) and (\ref{ram}) the following holds:

\begin{align*}
2g(\tilde{\calD}) - 2 + \sharp U &\leq (\deg p)\bigg(2g(\tcalC) - 2 + \sharp p(U)\bigg) + \sharp (S\cup T) + 16 \chi_{S}(\tcalC) + 12 H_{\tcalC} (\lambda) \\
&= (\deg p)\bigg(2g(\tcalC) - 2 + \sharp (S \cup T) +\sharp (p(U) \setminus (S\cup T)\bigg) + \sharp (S\cup T) + 16 \chi_{S}(\tcalC) + 12 H_{\tcalC} (\lambda) \\
&\leq 4\chi_{S \cup T}(\tcalC) + 4 \sharp (p(U)\setminus (S\cup T)) + \sharp (S\cup T) + 16 \chi_{S}(\tcalC) + 12 H_{\tcalC} (\lambda).
\end{align*}
We have to bound the number $\sharp (p(U)\setminus (S\cup T))$, but the points in the image of $U$ that are not in $S \cup T$ are precisely the zeros of the leading and constant terms in $F(X)$ and $G(X)$. Again we can estimate their number by looking at the height of these terms. We obtain that

\begin{align*}
\sharp (p(U) \setminus (S\cup T)) &\leq H_{\tcalC} \bigg( 2\uu - \ud \bigg) + H_{\tcalC} \bigg( \ud \bigg) +  2 H_{\tcalC} \bigg( 2 \uu - \ud \bigg) + H_{\tcalC} \bigg( \bigg( \uu \bigg)^{2} (4 - \lambda^{2}) + \lambda'^{2}\bigg) \\
				& \leq 4 \chi_{S}(\tcalC) + H_{\tcalC} \bigg( \bigg( \uu \bigg)^{2} (4 - \lambda^{2}) + \lambda'^{2}  \bigg) \\
				& \leq 8 \chi_{S}(\tcalC) + 4 H_{\tcalC}( \lambda).
\end{align*}
Taking this into account we can return to the previous inequality to finish our proof:
\begin{align*}
\chi_{U}(\tilde{\calD}) &\leq 4\chi_{S \cup T}(\tcalC) + 32 \chi_{S}(\tcalC) + 16 H_{\tcalC}( \lambda)  + \sharp S +\sharp  T +16 \chi_{S}(\tcalC) + 12 H_{\tcalC}( \lambda) \\
		& \leq  52 \chi_{S}(\tcalC) + 28 H_{\tcalC}( \lambda) + 5\sharp  T + \sharp  S \\
		& \leq  53 \chi_{S}(\tcalC) + 28 H_{\tcalC}( \lambda) + 5 \max\{0, 2g(\tcalC) - 2\}.
\end{align*}
\end{proof}
\end{lemma}

The next step in the proof of our main result is an application of a theorem by Corvaja and Zannier concerning the ``greatest common divisor'' of two rational functions on $\tcalC$ of the form $a-1$ and $b-1$ where $a$ and $b$ are units with respect to some specified finite set (in our case the set will be $U$). This result is the function field analogue of a theorem by the same authors obtained in the arithmetic case (see \cite{Corvaja2005}) and it should be remarked that this result is linked to Lang-Vojta's conjecture as pointed out by Silverman in \cite{Silverman2004}. We will need a corollary of this deep theorem as stated in \cite{Corvaja2005} (Corollary 2.3) which read as follows:

\begin{Th}[Corvaja and Zannier]\label{CZ}
Let $a,b \in \osu$ not both constant, and let $H:= \max\{H(a),H(b)\}$. Then
\begin{description}
\item[(i)] If $a,b$ are multiplicatively independent, we have
\begin{equation}\label{CZ1}
\sum_{v \notin S} \min \{ v(1-a), v(1-b) \} \leq 3 \sqrt[3]{2} (H(a),H(b)\chi(\calC))^{ \frac{1}{3}} \leq 3 \sqrt[3]{2}(H^{2}\chi(\calC))^{ \frac{1}{3}}
\end{equation}

\item[(ii)] If $a,b$ are multiplicatively dependent, let $a^{r} = \mu b^{s}$ be a generating relation. Then either $ \mu \neq 1$ and $\sum_{v \notin S} \min \{ v(1-a), v(1-b) \} = 0$ or $ \mu = 1$ and
\begin{equation}\label{CZ2}
\sum_{v \notin S} \min \{ v(1-a), v(1-b) \} \leq \min \bigg\{ \frac{H(a)}{\lvert s \rvert}, \frac{H(b)}{\lvert r \rvert} \bigg\} \leq \frac{H}{\max\{\rvert r \lvert, \lvert s \rvert \}}
\end{equation}
\end{description}
\end{Th}

We are going to apply this theorem for a suitable choice of the units $a$ and $b$: these units will be chosen in such a way that their heights will be ``close'' to the heights of $u_{1},u_{2}$ and such that the sum appearing in the statement of Theorem \ref{CZ} gives an upper bound for $\sum_{v \in \tcalD \setminus U} v(y)$. We begin by proving the following
\begin{lemma}\label{ab}
Let $(u_{1},u_{2},y)$ be a solution of equation (\ref{eq}) (recall that we are supposing that the leading and constant coefficients of $F,G$ are both non zero). Let $\tilde{\calD},U$ as before. Then there exist $U$-units $a,b \in \kappa(\tilde{\calD})$ such that
\begin{equation}\label{Hab}
\lvert \max\{H_{\tilde{\calD}}(a),H_{\tilde{\calD}}(b)\} - \max \{ H_{\tilde{\calD}}(u_{1}), H_{\tilde{\calD}}(u_{2})\} \rvert \leq 32 \cdot \chi_{S} (\tcalC) + 8 H_{\tcalC}(\lambda)
\end{equation}

\noindent and
\begin{equation}\label{vy-up}
\sum_{v \in \tilde{\calD} \setminus U} \min \{ v(1-a), v(1-b) \} \geq \frac{1}{4} \cdot \sum_{v \in \tilde{\calD} \setminus U} v(y).
\end{equation}

Moreover, $a = u_{1}\alpha^{-1}$, $b = u_{2}\beta^{-1}$ for suitable $ \alpha, \beta$ such that $F( \alpha) = G( \beta) = 0$.

\begin{proof}
Being the field $ \kappa(\tilde{\calD})$ defined as the splitting field for the polynomial $F(X)\cdot G(X)$ we can write the two polynomials as
\begin{align*}
F(X) &= \bigg( 2\uu - \ud \bigg) (X - \alpha) \cdot (X - \bar \alpha), \\
G(X) &= \bigg( 2 \uu - \ud \bigg)^{2} (X - \beta) \cdot (X - \bar \beta).
\end{align*}
We claim that the roots $ \alpha, \bar \alpha$ (resp. $\beta, \bar\beta$) of $F$ (resp. $G$) are $U$-units. This follows from the definition of $U$ (see Lemma \ref{cover}), because the leading and constant coefficients of the two polynomials are $U$-units. We consider now the following polynomials obtained from $F$ and $G$ dividing by $ \alpha \bar \alpha\big( 2\uu - \ud \big)$ and $\beta\bar\beta \big( 2 \uu - \ud \big)^{2}$ respectively, i.e. the polynomials

\begin{align*}
\overline F(X) &:= (X \alpha^{-1} -1)(X \bar\alpha^{-1} -1) \\
\overline G(X) &:= (X \beta^{-1} -1)(X \bar\beta^{-1} -1).
\end{align*}

Now, by Lemma \ref{AB}, the $U$-integer $y$ divides both $F(u_{1})$ and $G(u_{2})$, and hence it divides the polynomials $\overline F(u_{1})$ and $\overline G(u_{2})$ in the ring of $U$-integers. From this it follows that
\[
\sum_{v \in \tilde{\calD} \setminus U} \min\{v(u_{1} \alpha^{-1} - 1) + v(u_{1} \bar\alpha^{-1} - 1), v(u_{2} \bar\beta^{-1} - 1) + v(u_{2} \bar\beta^{-1} - 1)\} \geq \displaystyle{\sum_{v \in \tilde{\calD} \setminus U} v(y)}.
\]

We want to analyze the left-side term of the last inequality: observe that for every fourple of rational functions $W_{1},W_{2},Z_{1},Z_{2}$ one has (we omit the valuations)

\begin{align*}
\sum_{v} \min\{ W_{1} + W_{2}, Z_{1} + Z_{2} \} &\leq \sum_{v} \min \{ W_{1} , Z_{1} \} + \sum_{v} \min \{ W_{1} , Z_{2} \} + \sum_{v} \min \{ W_{2} , Z_{1} \} + \sum_{v} \min \{ W_{2} , Z_{2} \}  \\
&\leq 4 \sum_{v} \min \{ \tilde W , \tilde Z \},
\end{align*}

for suitable $\tilde W \in \{W_{1},W_{2}\}$ and $\tilde Z \in \{ Z_{1},Z_{2} \}$. In our case we obtain that there exist $U$-units $a \in \{u_{1} \alpha^{-1}, u_{1} \bar\alpha^{-1}\}$ and $b \in\{u_{2} \beta^{-1}, u_{2} \bar\beta^{-1}\}$ such that:
\[
4 \sum_{v \in \tilde{\calD} \setminus U} \min\{ v(a-1),v(b-1) \} \geq \sum_{v \in \tilde{\calD} \setminus U} v(y),
\]

proving (\ref{vy-up}). Next we want to prove that the heights of these $U$-units $a,b$ are ``close'' to the heights of $u_{1},u_{2}$. We observe that the difference appearing in the left side term of (\ref{Hab}) is bounded by the maximum of the $\tilde{\calD}$-heights of the roots of $F$ and $G$. Again we bound these heights by estimating their possible poles. It is then sufficient to observe that the poles of the roots $ \alpha, \bar \alpha$ (resp. $\beta , \bar \beta$) are either zeros of the leading coefficient or poles of the constant term of the polynomial $F$ (resp. G). Hence

\begin{align*}
\max\{ H_{\tilde{\calD}}( \alpha), H_{\tilde{\calD}} (\bar \alpha) \} &\leq H_{\tilde{\calD}}\bigg( 2\uu - \ud \bigg) + H_{\tilde{\calD}}\bigg(\ud\bigg) \\
								& \leq 4 H_{\tcalC}\bigg( 2\uu - \ud \bigg) + 4 H_{\tcalC}\bigg(\ud\bigg) \\
								& \leq 8 \chi_{s}(\tcalC).
\end{align*}
In the same way we get
\begin{align*}
\max\{ H_{\tilde{\calD}}( \beta), H_{\tilde{\calD}} (\bar \beta) \}  &\leq H_{\tilde{\calD}}\bigg( 2 \uu - \ud \bigg)^{2} + H_{\tilde{\calD}}\bigg[ \bigg( \uu \bigg)^{2} (4 - \lambda^{2}) + \lambda'^{2} \bigg] \\
								& \leq 4 H_{\tcalC}\bigg( 2 \uu - \ud \bigg)^{2} + 4 H_{\tcalC}\bigg[ \bigg( \uu \bigg)^{2} (4 - \lambda^{2}) + \lambda'^{2} \bigg] \\
								& \leq 32 \chi_{s}(\tcalC) + 8 H_{\tcalC}(\lambda).
\end{align*}
\end{proof}
\end{lemma}

In order to apply Theorem \ref{CZ} we need an upper bound for $\sum_{v \in \tilde{\calD} \setminus U} v(y)$ in terms of the heights of $u_{1},u_{2}$. This bound is obtained by an application of a theorem by U. Zannier in \cite{Zannier1993} which reads as follows:

\begin{Th}[Zannier]\label{Zannier}
Let $\tilde{\calD},U$ as before, $m\geq 2$ an integer, $ \theta_{1},\dots, \theta_{m}$ $U$-units such that no subsum of $ \theta_{1}+\dots+ \theta_{m}$ vanishes. Then the $U$-integer $ \theta_{1}+\dots+ \theta_{m}$ satisfies
\begin{equation*}
\sum_{v \in \tilde{\calD} \setminus U} v(\theta_{1}+\dots+ \theta_{m}) \geq H_{\tilde{\calD}}(\theta_{1}:\dots: \theta_{m}) - \binom{m}{2} \chi_{U}(\tilde{\calD}).
\end{equation*}
\end{Th}

We are going to apply this Theorem to the $U$-integer 
\[y = u_{1}^{2} + \lambda u_{1} + u_{2} + 1,
\]
using the fact that
\begin{align*}
H_{\tilde{\calD}} ( u_{1}^{2} : \lambda u_{1} : u_{2} : 1 ) & \geq \max\{ 2 H_{\tilde{\calD}} (u_{1}) , H_{\tilde{\calD}}(u_{1})+H_{\tilde{\calD}}( \lambda), H_{\tilde{\calD}}(u_2) \} \\ & \geq \max\{ H_{\tilde{\calD}} (u_{1}) , H_{\tilde{\calD}}(u_{2})\}.
\end{align*}
In particular, assuming that no subsum of the right term of equation (\ref{eq}) vanishes, we obtained the following

\begin{lemma}\label{vy-down}
For every solution $(y,u_{1},u_{2})$ of (\ref{eq}) such that no subsum of the right term vanishes, one has
\[
H_{\tilde{\calD}}(y) \geq \sum_{v \in \tilde{\calD} \setminus U} v(y) \geq \max\{ H_{\tilde{\calD}} (u_{1}) , H_{\tilde{\calD}}(u_{2})\} - 6 \chi_{U}(\tilde{\calD}).
\]
\end{lemma}

Now we put together this last inequality with the results of Lemma \ref{Hab} and we obtain that, for every solution of (\ref{eq}) there exist $U$-units $a,b$ such that 
\[
\sum_{v \in \tilde{\calD} \setminus U} \min\{ v(a-1),v(b-1) \} \geq \frac{1}{4} \bigg(\max\{ H_{\tilde{\calD}} (a) , H_{\tilde{\calD}}(b)\} 
- 6 \chi_{U}(\tilde{\calD}) - 32 \chi_{S} (\tcalC) -8 H_{\tcalC}(\lambda) \bigg)
\]

Using the fact that $\chi_{S}(\tcalC) \leq \chi_{U}(\tilde{\calD})$ we obtain that

\begin{equation}\label{vab}
\sum_{v \in \tilde{\calD} \setminus U} \min\{ v(a-1),v(b-1) \}\geq\frac{1}{4} \bigg( \max\{ H_{\tilde{\calD}} (a) , H_{\tilde{\calD}}(b)\}  - 38 \chi_{S} (\tcalC)-8 H_{\tcalC}(\lambda) \bigg).
\end{equation}

We can now apply Theorem \ref{CZ} to deduce the following

\begin{prop}\label{estimate}
Let $(y,u_{1},u_{2}) \in \os \times (\osu)^{2}$ be a solution of equation (\ref{eq}) such that no subsum of the right term vanishes, and the leading and constant term of the polynomials $F,G$ are not zero. Let $\tilde{\calD},U$ be as defined in Lemma \ref{cover} and $a,b \in \calO_{U}^{*}$ as defined in Lemma \ref{ab}. Then either

\begin{equation}\label{final}
\max\{ H_{\tcalC} (u_{1}) , H_{\tcalC}(u_{2})\} \leq 2^{12}  \bigg( 58 \chi_{S}(\tcalC) + 28 H_{\tcalC}(\lambda) \bigg) + 16 H_{\tcalC}(\lambda) 
\end{equation}

or $a,b$ verify a multiplicative dependence relation of the form
\[
a^{r} \cdot b^{s} = 1 \] 
for integers $(r,s) \in \Z^{2} \setminus \{ 0\}$ with
\begin{equation}\label{exp}
\max \{ \lvert r \rvert, \lvert s \rvert \} \leq 5.
\end{equation}

\begin{proof}
We suppose that inequality (\ref{final}) does not hold and we want to prove the dependence relation for $a,b$. In order to apply Corvaja and Zannier Theorem \ref{CZ} we are going to show that the left-hand side of (\ref{final}) is greater than the right-hand side of (\ref{CZ1}). Our starting point is
\[
\max\{ H_{\tcalC} (u_{1}) , H_{\tcalC}(u_{2})\} > 2^{12} \cdot \bigg( 58 \cdot \chi_{S}(\tcalC) + 28 H_{\tcalC}(\lambda) \bigg) + 16 H_{\tcalC}(\lambda).
\]
From Lemma \ref{cover} we know that
\[\chi_{U}(\tilde{\calD}) \leq 58 \cdot \chi_{S}(\tcalC) + 28 H_{\tcalC}(\lambda)
\]
and so we obtain that
\[
\max\{ H_{\tcalC} (u_{1}) , H_{\tcalC}(u_{2})\} > 2^{12} \chi_{U}(\tilde{\calD}) + 16 H_{\tcalC}(\lambda).
\]
Remember that our aim is to apply Theorem \ref{CZ} and so we need to work with the maximum of the heights of $a,b$. For this reason we apply (\ref{Hab}) which estimates the closeness of $H(u_{i})$ and $H(a),H(b)$ and we get 
\begin{align*}
\max\{H_{\tilde{\calD}}(a),H_{\tilde{\calD}}(b) \} &\geq \max \{ H_{\tilde{\calD}}(u_{1}), H_{\tilde{\calD}}(u_{2})\} - 32 \chi_{S} (\tcalC) - 8 H_{\tcalC}(\lambda)\\
& \geq \max \{ H_{\tilde{\calD}}(u_{1}), H_{\tilde{\calD}}(u_{2})\} - 32 \chi_{U} (\tilde{\calD}) -8 H_{\tcalC}(\lambda).
\end{align*}
From these last two inequalities and the fact that $H_{\tcalC} \leq H_{\tilde{\calD}}$, we obtain the lower bound
\begin{equation}\label{2^}
  \max\{H_{\tilde{\calD}}(a),H_{\tilde{\calD}}(b) \} \geq (2^{12} - 32) \chi_{U}(\tilde{\calD}) + 8 H_{\tcalC}(\lambda).
\end{equation}
In order to simplify the notation we put $H = \max\{H_{\tilde{\calD}}(a),H_{\tilde{\calD}}(b) \}$ and $\chi = \chi_{U}(\tilde{\calD})$. We claim that
\begin{equation}\label{toCZ}
\sum_{v \in \tilde{\calD} \setminus U} \min\{ v(a-1),v(b-1) \} > 3 \cdot 2^{ \frac{1}{3}} H^{ \frac{2}{3}} \chi^{ \frac{1}{3}}
\end{equation}
To prove the claim we observe that, from (\ref{vab}), it is enough to show that
\[
\frac{1}{4} \big(H - 38 \chi - 8 H_{\tcalC}(\lambda)\big)> 3 \cdot 2^{ \frac{1}{3}} H^{ \frac{2}{3}} \chi^{ \frac{1}{3}}.
\]
We define the function
\[
f(t) = \frac{1}{4} \big(t - 8 H_{\tcalC}(\lambda)\big) - 3  \cdot 2^{ \frac{1}{3}} t^{ \frac{2}{3}} \chi^{ \frac{1}{3}} - 38/4 \chi
\]
and we notice that our claim is equivalent to $f(H) > 0$. Now the function $f$ is an increasing function for $t \geq 2^{10}\chi + 8 H(\lambda)$, therefore it is enough to prove it for $H = (2^{12} - 32) \chi > 2^{10} \chi$. Hence the claim is equivalent to
\[
\frac{1}{4}(2^{12} - 32)\chi -  3 \cdot 2^{ \frac{1}{3}} (2^{12} - 32)^{ \frac{2}{3}} \chi > 38/4 \chi 
\]
With some algebraic manipulations one gets

\begin{align*}
&\frac{1}{4}(2^{12} - 32)\chi -  3 \cdot 2^{ \frac{1}{3}} (2^{12} - 32)^{ \frac{2}{3}} \chi  =\\
& 2^{ \frac{10}{3}} (2^{7}-1)^{ \frac{2}{3}} \bigg[ \frac{1}{4} 2^{ \frac{5}{3}}((2^{7} - 1)^{ \frac{1}{3}} - 3 \cdot 2^{ \frac{1}{3}} \bigg] \chi=\\
&40\cdot 6 \cdot \bigg[ 2^{ -\frac{1}{3}}\cdot 2^{ \frac{13}{6}}  - 3 \cdot 2^{ \frac{1}{3}} \bigg]\chi =\\
&40\cdot 6 \cdot \bigg[ 2^{ \frac{1}{3}}(2^{ \frac{5}{3}}  - 3) \bigg]\chi> 40\chi
\end{align*}
which proves the claim. Now we can apply Theorem \ref{CZ} which implies that $a,b$ verify a multiplicative dependence relation of the form $a^{r}b^{s} =1$ for some integers $r,s$ not both zero. The same Theorem gives the bound (\ref{CZ2}) and hence, together with (\ref{2^}) and (\ref{vab}), we obtain
\begin{equation*}
\frac{H}{\max\{\lvert r \rvert, \lvert s \rvert \}} > \frac{1}{4}H - 10 \chi > \frac{1}{5} H
\end{equation*}
Therefore we get $\max\{\lvert r \rvert, \lvert s \rvert \} \leq 5$, as desired.
\end{proof}
\end{prop}

The conclusion of Proposition \ref{estimate} gives us a multiplicative relation of dependence between $a,b$ instead of $u_1,u_2$. However this relation is guaranteed by Lemma 3.14 in \cite{Corvaja2005} which gives us the following result:

\begin{lemma}[\cite{Corvaja2005}]\label{relation}
In the previous notation, if a multiplicative relation of the form $a^r \cdot b^s = \mu$ holds for a constant $\mu \in \kappa$, then either one between $a$ and $b$ is constant or $u_1,u_2$ satisfy a multiplicative dependence relation of the same type.
\end{lemma}

Now we go back to Theorem \ref{mainth}: here we should take care of the constant term of the polynomial $G$ in a different way as in the constant case. In detail the vanishing of this term does not directly imply an explicit bound for the degree of the images $f(\calC)$ as in the split function field case; here we should apply again the whole machinery in order to explicitly find the unit $u_{1}$ and so reduce the problem to equation $y^{2} = \mu + u_{2} + 1$, which was already solved in the split case and gives the desired bound. For readability reasons we split the proof of Theorem \ref{mainth} in two cases: Lemma \ref{goodG} for the case in which the constant coefficient of $G$ is not zero, and Lemma \ref{Gconst} for the other case. Clearly the two lemmas together gives Theorem \ref{mainth}.

\begin{lemma}\label{goodG}
Suppose that the constant term of the polynomial $G$ does not vanish, i.e., with the notation of \ref{mainth}, every solution $(y,u_{1},u_{2}) \in \os \times (\osu)^{2}$ of equation (\ref{eq})
\[
y^{2} = u_{1}^{2} + \lambda u_{1} + u_{2} + 1
\]
satisfies also
\begin{equation}
\bigg( \uu \bigg)^{2} (4 - \lambda^{2}) + \lambda^{2} \bigg( \la \bigg)^{2} \neq 0.
\end{equation}

Then one of the following conditions holds:
\begin{description}
\item[(i)] a sub-sum on the right term of (\ref{eq}) vanishes;
\item[(ii)] $u_{1},u_{2}$ verify a multiplicative dependence relation of the form $u_{1}^{r} \cdot u_{2}^{s} = \mu$, where $ \mu \in \kappa$ is a scalar and $r,s,$ are integers, non both zeros such that $\max\{r,s\} \leq 5$;
\item[(iii)] the following bound holds:
\[
\hspace*{-1cm}\max\{H_{\tcalC}(u_{1}),H_{\tcalC}(u_{2})\} \leq 2^{12} \cdot \big( 58 \cdot \chi_{S}(\tcalC) + 28 H_{\tcalC}(\lambda) \big) + 16 H_{\tcalC}(\lambda).
\]
\end{description}
\end{lemma}
\begin{proof}
We start assuming that (i), (ii) and (iii) are not satisfied and we are going to find a contradiction. First of all we note that, if (i) is not satisfied, no subsum of (\ref{*}) can vanish. Moreover the polynomials $F$ and $G$ defined in \ref{AB} could not be constant because the vanishing of their leading coefficients would imply some multiplicative relation between $u_1$ and $u_2$ which is excluded by (ii). The same is true for the constant coefficient of $F$ (which is $u_2'/u_2$): it cannot be zero otherwise $u_2$ would be constant; moreover, by our assumptions, the same holds for the constant coefficient of $G$. Hence both $F$ and $G$ are non constant polynomials whose constant coefficients are not zero.

Since we excluded the case where the leading and constant coefficients of $F$ and $G$ vanish, we can apply (\ref{estimate}) and obtain a multiplicative relation between $a = u_1\alpha^{-1}$ and $b = u_1\beta^{-1}$; this follows from the fact that inequality (\ref{final}) is excluded by (iii). From this relation, applying (\ref{relation}), we get that either $a$ or $b$ is constant or $u_1$ and $u_2$ verify a multiplicative relation of the same type. The former case would imply that the height of $u_1$ (or $u_2$) would be the same as the height of $\alpha$ (resp. $\beta$) so it would be lesser or equal than $8 \chi_S(\tcalC)$ (resp. $32 \chi_S(\tcalC)$); but this contradicts our assumption that (iii) does not hold and hence it is excluded. The latter case is precisely (ii) that was assumed to be false. In both cases we get a contradiction and this concludes the proof.

\end{proof}

\begin{lemma}\label{Gconst}
Suppose that the constant term of the polynomial $G$ vanishes, i.e., with the notation of \ref{mainth}, every solution $(y,u_{1},u_{2}) \in \os \times (\osu)^{2}$ of equation (\ref{eq}) satisfies also
\begin{equation}\label{Gcoeff}
\bigg( \uu \bigg)^{2} (4 - \lambda^{2}) + \lambda^{2} \bigg( \la \bigg)^{2} =0.
\end{equation}

Then one of the following conditions holds:
\begin{description}
\item[(i)] $(y,u_{1},u_{2})$ satisfy an equation whose solutions verify conclusions of Theorem \ref{mainth}.
\item[(ii)] $u_{1},u_{2}$ verify a multiplicative dependence relation of the form $u_{1}^{r} \cdot u_{2}^{s} = \mu$, where $ \mu \in \kappa$ is a scalar and $r,s,$ are integers, not both zero such that $\max\{r,s\} \leq 5$;
\end{description}
\end{lemma}
\begin{proof}

The first trivial case is the case in which $\lambda$ is constant which is excluded since we are assuming that the threefold defined by $\lambda$ is not trivial. The second case is the case in which $\lambda$ is a non constant $S$-unit. In this case we obtain, in the ring $\os$, the following identity (here we can enlarge $S$ so that it contains every point for which $\lambda = 2$):

\begin{equation}\label{Gconst_eq}
\bigg( \uu \bigg)^{2} = -  \frac{\lambda^{2}}{4 - \lambda^{2}} \bigg( \la \bigg)^{2}.
\end{equation}

Now we observe that, by Dirichlet Unit Theorem, the ring $\osu$ is finitely generated modulo constants, so every $u_1 \in \osu$ is of the form $\mu \cdot v_1^{a_1} \cdots v_{h}^{a_h}$ for some $\mu \in \kappa$ and $v_1,\dots,v_h \in \osu$. Therefore we have

\[
\uu = \sum_{i=1}^h a_i \frac{v_i'}{v_i}.
\]

Being $\lambda \in \osu$ the right-hand side of equation (\ref{Gconst_eq}) could also be expressed in the $v_i$ and their derivatives; in particular \ref{Gconst_eq} becomes an equation in the unknown $a_i$ and this equation will have a unique (for given $\lambda$ and $S$) solution in the $a_i$. Hence $u_1$ will be uniquely determined up to a constant factor and therefore its height will be a constant. So we can assume that $u_{1} = a f$ for a constant $a \in \kappa$ and a fixed $S$-unit $f$. This leads to consider equation $y^{2} = a^{2} f^{2} + \lambda a f + u_{2} + 1$. We claim that this case gives (i). The claim follows from a repetition of all the considerations done until now for equation (\ref{eq}): we obtain the same estimates with different polynomials $\tilde F, \tilde G$. Again we look at the vanishing of the constant and leading coefficients and this time we found that the case in which the constant coefficient of the new polynomial $\tilde G$ vanishes gives us either $u_{1} = 0$ or $u_{1} = f$ where $a=1$. In both cases this reduces the problem to the equation $y^{2}= \mu + u_{2} + 1$, where $\mu$ is now fixed, which has already been treated in the split function field case and gives (i).
The case in which the constant term of $\tilde G$ is not zero is precisely one of the cases of (i) and this concludes the proof of the claim.
\end{proof}

Finally we prove Theorem \ref{th:deg} using the previous Theorem.

\begin{proof}[Proof of Theorem \ref{th:deg}]
As in Lemma \ref{eql}, after possibly passing to a double cover of $\tcalC$, a section
\[
\sigma : \tcalC \setminus S \to X
\]
will be of the form
\[
\sigma : P \mapsto (u_1(P), y(P), P)
\]
\noindent where the $S$-unit $u_1$ and the $S$-integer $y$ verify equation (\ref{eq}) for a $S$-unit $u_2$. In this setting we can apply Theorem \ref{mainth} and conclude that one of (i),(ii),(iii) holds. Let us analyze every case.
\begin{itemize}
\item In the first case (i) we have that some sub-sum of $u_{1}^{2} + \lambda(P) u_{1} + u_{2} + 1$ will vanish. Hence $\sigma(\calC)$ is either a line or a conic (recall that being $\lambda$ non constant its height is at least one).
\item In the second case (ii) we have a multiplicative relation between the two $S$-units of the form $u_1^{r} = u_2^{s} \cdot \mu$ for a scalar $\mu \in \kappa$ and two integers $r,s$ with absolute value lesser or equal than 5. From this it follows that $\deg \sigma(\calC) \leq H_{\tcalC} (u_1) + H_{\tcalC} (y) \leq 20$.
\item In the last case (iii) we have $\max\{H_{\tcalC}(u_{1}),H_{\tcalC}(u_{2})\}$ is bounded above by $2^{12} \cdot \big( 58 \cdot \chi_{S}(\tcalC) + 28 H_{\tcalC}(\lambda) \big) + 16 H_{\tcalC}(\lambda)$.
\end{itemize}
In both cases this imply that there exist constants $C_1 \leq 2^{14} \cdot 58$ and $C_2 \leq 2^{14} \cdot 28$ that verify the conclusion of the Theorem.
\end{proof}

\bibliographystyle{alphanum}	
	\bibliography{../Mendeley}

\Addresses

\end{document}